\documentclass[11pt]{article}
\usepackage[utf8]{inputenc}
\usepackage{amsmath,amssymb,amsfonts,amsthm,enumitem,xy,pstricks,pst-node,graphicx,comment}
\xyoption{all}

\usepackage[initials,nobysame]{amsrefs}

\BibSpec{collection.article}{%
	+{}  {\PrintAuthors}				{author}
	+{,} { \textit}                  	{title}
	+{.} { }                         	{part}
	+{:} { \textit}                  	{subtitle}
	+{,} { \PrintContributions}      	{contribution}
	+{,} { \PrintConference}         	{conference}
	+{}  {\PrintBook}                	{book}
	+{,} { }                         	{booktitle}
	+{,} { }						 	{series}
	+{,} { }						 	{publisher}
	+{,} { }						 	{address}
	+{,} { \PrintDateB}              	{date}
	+{,} { pp.~}                     	{pages}
	+{,} { }                         	{status}
	+{,} { \PrintDOI}                	{doi}
	+{,} { available at \eprint}     	{eprint}
	+{}  { \parenthesize}            	{language}
	+{}  { \PrintTranslation}        	{translation}
	+{;} { \PrintReprint}            	{reprint}
	+{.} { }                         	{note}
	+{.} {}                          	{transition}
	+{}  {\SentenceSpace \PrintReviews} {review}
}

\title{On Gromov Width under $C^0$ Deformations}
\author{Spencer Cattalani\thanks{Partially supported by NSF grant DMS 1901979 and the Simons Foundation}}
\date{\today}

\newtheorem{thm}{Theorem}
\newtheorem{prp}[thm]{Proposition}
\newtheorem{lmm}[thm]{Lemma}   
 
\newtheorem{eg}[thm]{Example}
\newtheorem{conj}[thm]{Conjecture}

\theoremstyle{definition}
\newtheorem{dfn}[thm]{Definition}

\theoremstyle{remark}
\newtheorem{rmk}[thm]{Remark}

\theoremstyle{remark}

\def\BE#1{\begin{equation}\label{#1}}
\def\EE{\end{equation}}
\def\e_ref#1{(\ref{#1})}

\def\lra{\longrightarrow}

\def\R{\mathbb R}
\def\Z{\mathbb Z}

\def\ev{\textnormal{ev}}

\def\id{\textnormal{id}}

\def\pt{\textnormal{pt}}

\def\compJs{\mathcal{J}_c(\omega)}
\def\tamedJs{\mathcal{J}_\tau(\omega)}
\def\unicompJs{\mathcal{J}^\infty_c(\omega)}
\def\unitamedJs{\mathcal{J}^\infty_\tau(\omega)}

\def\compSs{\mathcal{B}_c(X)}
\def\tamedSs{\mathcal{B}_\tau(X)}
\def\unicompSs{\mathcal{B}^\infty_c(X)}
\def\unitamedSs{\mathcal{B}^\infty_\tau(X)}

\def\tamingconstant{\delta_S}

\def\d{\textnormal{d}}

\def\gromovwidth{c_{G}}
\def\ev{\textnormal{ev}}
\def\modulispace{\overline{\mathfrak{M}}_{g,k}(A;J)}
\def\specificmodulispace#1{\mathfrak{M}_{0,1}\big([S^2\!\times\!\pt];#1\big)}

\begin{document}

\maketitle

\begin{abstract}
We construct a uniformly bounded symplectic structure on $S^2 \times \R^{4}$ admitting embeddings by arbitrarily large balls. This provides a counterexample to a recent conjecture of Savelyev. We then prove the conjecture holds for a wide class of examples, generalizing a result by Savelyev. Along the way, we clarify some aspects of pseudoholomorphic curve theory in non-compact manifolds.
\end{abstract}

\section{Introduction}

Gromov's non-squeezing theorem \cite[Section 0.3.A]{Gromov1985} is one of the seminal results in symplectic geometry.
It states that, if $B^{2n}_R$ embeds symplectically into $Z^{2n}_r$, then $R \leq r$,
where
$$B_R^{2n}, Z_r^{2n} \subset (\R^{2n},\omega_{std})$$ are a ball of radius $R$ and an $r$-neighborhood of a symplectic hyperplane, respectively. One can ask to what degree this rigidity persists if the symplectic form on the target is changed. That is the focus of the present article.\\

We first recall some definitions. If $(X^{2n}, \omega)$ is a symplectic manifold, its \textit{Gromov width} $\gromovwidth(\omega)$ is the supremal $\pi R^2$ such that there is a symplectic embedding of $B^{2n}_R$ into $(X,\omega)$.
If $\omega$ is a $2$-form on a Riemannian manifold $(X,g)$, its \textit{comass} $||\omega||$ is the supremal value of $\omega$ over simple bivectors $v \wedge w$ of unit area. We say that $\omega$ is \textit{bounded} if $||\omega||$ is finite.\\

Throughout, we take the standard symplectic form $\omega_0$ on $S^2 \times \R^{2n}$ to be the product of a symplectic form on $S^2$ (which is compatible with its integrable complex structure) with total area $1$ and the standard linear symplectic form on $\R^{2n}$. We take the standard complex structure $J_0$ on $S^2 \times \R^{2n}$ to be that of $\mathbb{CP}^1 \times \mathbb{C}^n$ and the standard metric to be $\omega(-,J_0-)$. In particular,
\begin{equation}\label{comass_ineq}
||\omega_0|| = 1 \quad \textnormal{and} \quad \big|\langle \omega , [S^2 \!\times\! \pt] \rangle  \big| \leq || \omega||
\end{equation}
for every symplectic form $\omega$ on $S^2\!\times\!\R^{2n}$. Savelyev recently posed the following conjecture.

\begin{conj}[{{\cite[Conjecture~1]{Savelyev2025}}}]\label{sav_conj}
Let $n \geq 1$ and $\omega$ be a symplectic form on $S^2 \times \R^{2n}$. If $|| \omega - \omega_0||$ is finite, then
\begin{equation}\label{conj_ineq}
\gromovwidth(\omega) \leq 1 + || \omega - \omega_0||.
\end{equation}
\end{conj}

Gromov showed in \cite[Section~0.3.A, Remarks]{Gromov1985} that Conjecture~\ref{sav_conj} fails if $S^2 \times \R^{2n}$ is replaced with~$Z^{2n}_r$. We first modify Gromov's example to provide a counterexample to Conjecture~\ref{sav_conj}.

\begin{thm}\label{counterexample_thm}
Let $n \geq 2$. There exists a symplectic form $\omega$ on $S^2 \times \R^{2n}$ such that $|| \omega - \omega_0||$ is finite, but
\begin{equation*}
\gromovwidth(\omega) = \infty.
\end{equation*}
\end{thm}

Although Theorem~\ref{counterexample_thm} shows that Conjecture~\ref{sav_conj} is false in general, many interesting special cases can be proved. The standard technique for proving results of this type, due to Gromov \cite[Theorem~0.3.A]{Gromov1985}, is to use a continuity argument to construct a pseudoholomorphic curve through a given point. That is, the Gromov width can be bounded by defining a nonzero Gromov-Witten invariant with a point insertion. This is more difficult in the non-compact setting and Theorem~\ref{counterexample_thm} shows that it cannot always be done. In Sections~\ref{uniform_sec} and \ref{conv_sec}, we study two general settings in which pseudoholomorphic curves are well-behaved and prove some basic statements in each setting. Our approach in these sections is based heavily on the ideas presented by Sikorav in \cite{Sikorav94}. The setting of Section~\ref{uniform_sec} is that of \textit{uniformly tamed} almost complex structures. We define a \textit{strong deformation} of symplectic forms (Definition~\ref{strongly_uniform_dfn}) and prove the following theorem.

\begin{thm}\label{deformation_thm}
If $\omega$ is strongly deformation equivalent to $\omega_0$, then
\begin{equation}\label{deformation_ineq}
\gromovwidth(\omega) \leq \langle \omega , [S^2 \!\times\! \pt] \rangle.
\end{equation}
\end{thm}

\begin{rmk}\label{bound_rmk}
Along with (\ref{comass_ineq}), (\ref{deformation_ineq}) gives $\gromovwidth(\omega) \leq ||\omega||$, which implies (\ref{conj_ineq}). This is not too surprising. In fact, the proof of \cite[Theorem~1.1]{Savelyev2025}, which is a special case of Conjecture~\ref{sav_conj}, yields a bound of the form (\ref{deformation_ineq}), but then the triangle inequality is applied superfluously, which weakens it to (\ref{conj_ineq}).
\end{rmk}

Alternatively, one can prove special cases of Conjecture~\ref{sav_conj} using convexity instead of uniform taming. This is the route taken in \cite{Savelyev2025}, which introduces the notion of \textit{$J$-holomorphic trap}. Unfortunately, due to their reliance on the positivity of intersections, traps are of limited use in dimensions above $4$. Theorem~\ref{convex_thm} below generalizes \cite[Theorem~1.1]{Savelyev2025} by removing the assumption that $\omega$ can be deformed to $\omega_0$ via a special family of symplectic forms. The resulting bound is also stronger than (\ref{conj_ineq}), in the same way as explained in Remark~$\ref{bound_rmk}$.

\begin{thm}\label{convex_thm}
Let $\omega$ be a symplectic form on $S^2 \!\times\! \R^2$. If the fibers of the projection $\pi\!: S^2 \!\times\! \R^2 \lra \R^2$ are $\omega$-symplectic submanifolds outside of a compact subset $K \subset S^2 \!\times\!\R^2$, then
$$\gromovwidth(\omega) \leq \langle \omega , [S^2 \!\times\! \pt] \rangle.$$
\end{thm}

Although Theorem~\ref{convex_thm} could be proved using the $J$-holomorphic traps introduced in \cite{Savelyev2025}, we have opted to use a more general, classical notion of convexity (essentially found in \cite[Section~0.4]{Gromov1985} and \cite[Section~4.1]{Sikorav94}) which works in any dimension. We believe this clarifies the argument and highlights what the essential difficulties in generalizing Theorem~\ref{convex_thm} to higher dimensions are.\\

We recall that a Riemannian manifold is said to have \textit{bounded geometry} if it has a uniform lower bound on the injectivity radius and uniform upper and lower bounds on the curvature. Both conditions are clearly satisfied for the standard metric on $S^2 \!\times\! \R^{2n}$.
The common feature of the settings of Sections~\ref{uniform_sec} and \ref{conv_sec} is that the usual evaluation map
\begin{equation}\label{evaluation_map}
\ev\!: \modulispace \lra X^{k}
\end{equation}
from the moduli space of stable $J$-holomorphic curves of genus $g$, degree $A \in H_2(X;\Z)$, and with $k$ marked points is proper if $J$ is an almost complex structure on $X$ and either
\begin{enumerate}[label=(\alph*),leftmargin=*]
\item $J$ is uniformly tamed by a symplectic form with respect to a Riemannian metric $g$ on $X$ with bounded geometry, or
\item $J$ is convex and tamed by a symplectic form;
\end{enumerate}
see Definitions~\ref{uniformly_tamed_dfn} and \ref{convex_dfn}, respectively. In the first case, the properness of (\ref{evaluation_map}) is the subject of \cite[Theorem~5.2.3]{Sikorav94}. In the second case, it follows immediately from Gromov's compactness theorem \cite[Theorem~1.5.B]{Gromov1985} and Definition~\ref{convex_dfn} in the present paper. The techniques in Section~\ref{uniform_sec} actually suffice to define Gromov-Witten invariants for a class of non-compact symplectic manifolds and show their deformation invariance. This is discussed further in Remark~\ref{GW_rmk}.\\

For our applications, we need to consider only the moduli space~$\overline{\mathfrak{M}}_{0,1}([S^2\!\times\!\pt];J)$ with $J$ tamed by a symplectic form. In this case, $\overline{\mathfrak{M}}_{0,1}([S^2\!\times\!\pt];J)$ is equal to the subspace  $\mathfrak{M}_{0,1}([S^2\!\times\!\pt];J)$ consisting of maps with smooth domains. This is because $[S^2\!\times\!\pt]$ generates $H_2(S^2\!\times\!\R^{2n};\Z)$, so no bubbling can occur. Therefore, the following proposition, which is key to the proofs of Theorems~\ref{deformation_thm} and \ref{convex_thm}, holds.

\begin{prp}\label{proper_prop}
Let $\omega$ and $J$ be a symplectic form and an almost complex structure on $S^2\!\times\!\R^{2n}$, respectively. If $J$ is either uniformly tamed by $\omega$ with respect to the standard metric or convex and tamed by $\omega$,
then the evaluation map
\begin{equation}\label{specific_evaluation_map}
\ev\!: \specificmodulispace{J} \lra S^2 \!\times\! \R^{2n}
\end{equation}
is proper.
\end{prp}

The paper is organized as follows. Section~\ref{counterexample_sec} contains the proof of Theorem~\ref{counterexample_thm}. Section~\ref{uniform_sec} concerns uniformly tamed almost complex manifolds and contains the proof of Theorem~\ref{deformation_thm}. Section~\ref{conv_sec} concerns convex almost complex manifolds and contains the proof of Theorem~\ref{convex_thm}.\\

\textbf{Acknowledgements.} The author would like to thank Y.~Savelyev his encouragement to write this note, A.~Zinger for his comments on an earlier draft, which greatly improved the exposition, and the anonymous referee for helpful comments.

\section{Counterexample to Conjecture~\ref{sav_conj}}\label{counterexample_sec}
We now establish the $n=2$ case of Theorem~\ref{counterexample_thm}. Taking the product with $(\R^{2n},\omega_{std})$ establishes the full statement of Theorem~\ref{counterexample_thm} and yields a counterexample to Conjecture~\ref{sav_conj} for each $n \geq 2$.\\

Associate $S^2$ with the unit sphere in $\R^3$ and equip $\R^3 \times \R^3$ with coordinates $(x_1,x_2,x_3,y_1,y_2,y_3)$ so that $\omega_{std} = \sum dx_i \wedge dy_i$. Let $0 < a < b$ and $f\!: \R \lra (a,b)$ be a diffeomorphism with uniformly bounded derivative. The map
$$\varphi\!: S^2 \times \R^4 \lra \R^3 \times \R^3, \quad \varphi(\sigma,t,x,y,z) = \big(f(t)\sigma, x, y, z \big),$$
is then a diffeomorphism onto an open subset $U\!\times\!\R^3 \subset \R^3 \times \R^3$ with a uniformly bounded derivative. Thus, $\omega := \varphi^* \omega_{std}$ is a bounded symplectic form on $S^2 \times \R^4$.
Since $B^3_\varepsilon(p) \subset U$ for some $p \in \R^3$ and $\varepsilon > 0$,
$$\gromovwidth(\omega) := \gromovwidth(\omega_{std}|_U) = \infty.$$
As noted in \cite[Section~0.3, Remarks]{Gromov1985}, the equality above is achieved via the symplectic embeddings
$$B_R^6 \lra U\!\times\!\R^3, \quad (x_1,x_2,x_3,y_1,y_2,y_3) \mapsto \big(p+\varepsilon/R(x_1,x_2,x_3),R/\varepsilon(y_1,y_2,y_3)\big),$$
with $R > 0$.

\begin{rmk}
The above construction is motivated by the canonical symplectic form on the total space of the cotangent bundle of any smooth manifold having infinite Gromov width. The diffeomorphism~$\varphi$ identifies $S^2 \times \R^4$ with $T^*U$ so that the pullback of the canonical form is bounded.
\end{rmk}

\begin{rmk}
Equip $\R^3 \! \times \! \R$ with coordinates $(x_1,x_2,y_1,y_2)$ so that $\omega_{std} = \sum dx_i \wedge dy_i$. With $f$ as above, the map
$$\varphi\!: S^2 \! \times \! \R^2 \lra \R^3 \! \times \! \R, \quad \varphi(\sigma,t,z) = \big( f(t)\sigma,z\big),$$
is still a diffeomorphism onto an open subset $U \!\times\! \R^2 \subset \R^2 \! \times\!\R^2$ with a uniformly bounded derivative.
Thus, $\omega := \varphi^*\omega_{std}$ is a bounded symplectic form on $S^2 \! \times \! \R^2$.
Since $\langle\omega, [S^2 \!\times\! \pt] \rangle = 0$, for no almost complex structure $J$ tamed by $\omega$ does there exist a $J$-holomorphic sphere in the class $[S^2 \!\times\! \pt]$; this precludes the standard argument for non-squeezing. However, $||\omega - \omega_0|| \geq \pi b^2 - 1$ for $b\geq 1$ and $U\!\times \R^2 \subset Z^4_b$, so \hbox{$\gromovwidth(\omega) \leq \pi b^2$.} Thus, this particular construction does not provide a counterexample to the $n=1$ case of Conjecture~\ref{sav_conj}.
\end{rmk}

\section{Uniform Taming}\label{uniform_sec}
In this section, we prove Theorem~\ref{deformation_thm}. Our main tool is the use of \textit{uniformly tamed} almost complex manifolds, which is due to \cite[Section~4.1]{Sikorav94}. These are a natural generalization of compact almost complex manifolds and their study parallels the standard theory. In particular, this class is preserved under uniformly small perturbations of the almost complex structure, which allows the standard genericity results to carry over. This is in contrast to \textit{convex} almost complex manifolds, which are the focus of Section~\ref{conv_sec}. Working with uniformly tamed almost complex manifolds requires some quantitative results, which can be of interest even in the compact setting. This is explained further in Remark~\ref{interesting_bounds_rmk}.\\

Given an endomorphism $A$ of a normed vector space $(V,|\cdot|)$, we define the operator norm
$$||A|| := \sup_{v \in V-0} \frac{|Av|}{|v|}\, .$$
Given an endomorphism $A$ of a normed vector bundle, we define the (uniform) operator norm~$||A||_\infty$ as the supremum of the pointwise operator norms.

\begin{dfn}[cf.~{{\cite[Definition~4.1.1]{Sikorav94}}}]\label{uniformly_tamed_dfn}
Let $(X,g)$ be a Riemannian manifold. A bounded symplectic form $\omega$ on $X$ \textit{uniformly tames} an almost complex structure $J$ on $X$ if there exists $\delta > 0$ such that
\begin{equation}\label{uniform_taming_eqns}
\omega(v,Jv) \geq \delta |v|^2 \quad \forall v \in TX.
\end{equation}
If in addition $\omega(J-,J-) = \omega(-,-)$, then $J$ is \textit{uniformly compatible} with $\omega$.
\end{dfn}
For example, on a compact symplectic manifold, every tamed (resp. compatible) almost complex structure is uniformly tamed (resp. uniformly compatible). A tamed almost complex structure agreeing with a uniformly tamed almost complex structure outside of a compact set is itself uniformly tamed. The standard almost complex structure $J_0$ on $S^2 \!\times\! \R^{2n}$ is uniformly compatible with~$\omega_0$.
Given a Riemannian manifold $(X,g)$ equipped with a symplectic form $\omega$, we denote the spaces of almost complex structures which are tamed, compatible, uniformly tamed, and uniformly compatible by $\tamedJs$, $\compJs$,  $\unitamedJs$, and  $\unicompJs$, respectively.

\begin{lmm}\label{contractibility_lmm}
Let $(X,g)$ be a Riemannian manifold and $\omega$ be a bounded symplectic form on $X$. If $\unicompJs \neq \emptyset$, then both $\unicompJs$ and $\unitamedJs$ are contractible.
\end{lmm}

\begin{proof}
Fix an almost complex structure $J_\bullet \in \unicompJs$. It defines a Riemannian metric $g'$ on $X$ by $g'(v,w) = \omega(v,J_\bullet w)$. By (\ref{uniform_taming_eqns}) and the boundedness of $\omega$, $g'$ is uniformly bi-Lipschitz to $g$. Therefore, an almost complex structure is uniformly bounded with respect to $g'$ if and only if it is so with respect to $g$. For the rest of the proof, we will work exclusively with $g'$.\\

Denote by $\tamedSs$ the space of continuous bundle maps $S\!: TX \lra TX$ covering the identity such~that
\begin{equation}\label{Cayley_transform_eqn}
J_\bullet S + S J_\bullet = 0 \quad \textnormal{and} \quad \big|\big|S|_{T_x X}\big|\big| < 1 \quad \forall x \in X.
\end{equation}

Let $\compSs \subset \tamedSs$ be the subset of $S \in \tamedSs$ which satisfy
$$\omega(Sv,w) = -\omega(v, Sw) \quad \forall~ v,w \in T_x X, \, x \in X.$$

Define
$$\unitamedSs := \bigl\{S \in \tamedSs \,\,\big|\, ||S||_\infty < 1 \bigr\} \quad \textnormal{and} \quad \unicompSs := \compSs \cap \unitamedSs.$$
We note that each of the sets $\tamedSs$, $\compSs$, $\unitamedSs$, and $\unicompSs$ is convex and therefore contractible. By \cite[Proposition~1.1.6]{Audin1994}, the map
$$\Phi\!: \tamedSs \lra \tamedJs, \quad \Phi(S) = J_\bullet \circ (\id + S)(\id - S)^{-1},$$
is a homeomorphism and restricts to a homeomorphism from $\compSs$ to $\compJs$.\\

We now show that $\Phi$ restricts to a homeomorphism from $\unitamedSs$ to $\unitamedJs$ and thus to a homeomorphism from $\unicompSs$ to $\unicompJs$. Let $S \in \tamedSs$. The first statement follows from the inequality
\begin{equation}\label{Cayley_inequality}
\frac{1-||S||^2_\infty}{4} \leq \inf_{v \in TX - X} \tamingconstant(v) \leq 1 - ||S||^2_\infty\,.
\end{equation}
where $\tamingconstant(v) := \omega(v,\Phi(S)v)/|v|^2$ for nonzero $v \in TX$. The map $\id - S$ is invertible by the inequality in (\ref{Cayley_transform_eqn}), so every $v \in TX$ has the form $(\id - S)w$ for some $w \in TX$. A quick calculation shows that
\begin{equation}\label{easier_K_S_ineq}
\tamingconstant(v) = \frac{|w|^2 - |Sw|^2}{|w - Sw|^2} \geq
\frac{|w|^2 - |Sw|^2}{\big(|w| + |Sw|\big)^2} \geq 
 \frac{\big(1 - ||S||^2_\infty\big)|w|^2}{\big(1 + ||S||_\infty\big)^2|w|^2}\,.
\end{equation}
Since $||S||_\infty \leq 1$, this yields the first inequality in~(\ref{Cayley_inequality}). The anticommutativity expressed in (\ref{Cayley_transform_eqn}) and the fact that $J_\bullet$ is an isometry yield
$$\tamingconstant(J_\bullet v) =  \frac{|w|^2 - |Sw|^2}{|w + Sw|^2}\,.$$
Along with the equality in (\ref{easier_K_S_ineq}), this implies that either $\tamingconstant(v)$ or $\tamingconstant(J_\bullet v)$ is at most $1 - (|Sw|/|w|)^2$. Taking the infimum yields the second inequality in (\ref{Cayley_inequality}).
\end{proof}

\begin{rmk}\label{interesting_bounds_rmk}
The bounds in (\ref{Cayley_inequality}) are geometrically interesting in their own right. Call an almost complex structure which satisfies (\ref{uniform_taming_eqns}) for a given $\delta >0$ \textit{$\delta$-tamed}. By (\ref{Cayley_inequality}), if the metric is induced by a compatible almost complex structure, then the inclusion of the set of $\delta$-tamed almost complex structures into the set of $\delta/4$-tamed almost complex structures is null-homotopic. That is, the space of tamed $J$ is ``linearly'' contractible.
\end{rmk}


\begin{dfn}
Let $(X,g)$ be a Riemannian manifold. A continuous path
$$(A_t\!:TX\lra TX)_{t \in [0,1]}$$
of continuous bundle maps covering the identity is a \textit{bounded deformation of operators} if $A_0 = \id$ and the operator norms $||A_t||_\infty$ and $||A^{-1}_t||_\infty$ are uniformly bounded in $t$.
\end{dfn}

\begin{dfn}\label{strongly_uniform_dfn}
Let $(X,g)$ be a Riemannian manifold. A continuous path of symplectic forms $(\omega_t)_{t \in [0,1]}$ on $X$ is a \textit{strong deformation} of $\omega_0$ if there is a bounded deformation of operators $(A_t)_{t \in [0,1]}$ such that $$\omega_t(-,-) = \omega_0(A_t-,A_t-) \quad \forall\,t \in [0,1].$$
If there is a strong deformation connecting two symplectic forms, they are called \textit{strongly deformation equivalent}.
\end{dfn}

We note that strong deformation equivalence, like uniform taming, is a $C^0$ notion. A conformal change of metric can change the comass norm, but does not change the operator norm. In this way, strong deformation equivalences are stronger than $C^0$-small deformation equivalences, and have better compactness (or coercivity) properties. 

\begin{prp}\label{existence_of_curve_prp}
Let $\omega$ be strongly deformation equivalent to $\omega_0$ on $S^2 \!\times\! \R^{2n}$. Then, $\unicompJs \neq \emptyset$. For any almost complex structure $J$ on $S^2 \!\times\! \R^{2n}$ which is uniformly tamed by $\omega$ and any point $x \in S^2 \!\times\! \R^{2n}$, there is a $J$-holomorphic sphere representing the homology class $[S^2 \!\times\! \pt]$ passing through $x$.
\end{prp}

\begin{proof}
The product complex structure $J_0$ on $S^2 \!\times\! \R^{2n}$ is regular for rational curves in $[S^2 \!\times\! \pt]$ by~\cite[Lemma~3.3.1]{McDuffSalamon2012}. The evaluation map
$$\ev\!: \specificmodulispace{J_0} \lra S^2 \!\times\! \R^{2n}$$
is proper and has degree $1\!\! \mod 2$ in Borel-Moore homology.\\

Let $(A_t)_{t \in [0,1]}$ be a bounded deformation of operators such that $\omega_t := \omega_0(A_t -,A_t -)$ is symplectic for each $t \in [0,1]$ and $\omega_1 = \omega$. For each $t \in [0,1]$, $J_t := A_t J_0 A_t^{-1}$ is an almost complex structure uniformly compatible with $\omega_t$. By Lemma~\ref{contractibility_lmm}, every almost complex structure $J$ uniformly tamed by $\omega$ is connected by a path of uniformly tamed almost complex structures to $J_1$ and therefore to $J_0$. Proposition~\ref{proper_prop} and \cite[Theorems~3.1.6 and 3.1.8]{McDuffSalamon2012} imply that a for generic $J$, a generic path between $J$ and $J_0$ lifts to a proper bordism between the evaluation maps
$$\ev\!: \specificmodulispace{J_0} \lra S^2 \!\times\! \R^{2n} \quad \textnormal{and} \quad \ev\!: \specificmodulispace{J} \lra S^2 \!\times\! \R^{2n}.$$
As the evaluation map for $J_0$ has degree $1 \!\!\mod 2$, so does that for $J$.
\end{proof}

\begin{proof}[{{Proof of Theorem~\ref{deformation_thm}}}]
Let $\iota\!: (B^{2n+2}_R,\omega_{std}) \lra (S^2 \!\times\! \R^{2n},\omega)$
be a symplectic embedding. By Proposition~\ref{existence_of_curve_prp}, there is an almost complex structure $J$ on $S^2 \!\times\! \R^{2n}$ which is uniformly tamed by $\omega$. For $\varepsilon > 0$, let $J_\varepsilon$ be a tamed almost complex structure on $S^2 \!\times\! \R^{2n}$ which is equal to the pushforward of the standard integrable complex structure on $\iota(B^{2n+2}_{R-\varepsilon})$ and to~$J$ outside of a compact set. The structure $J_\varepsilon$ is uniformly tamed by $\omega$. By Proposition~\ref{existence_of_curve_prp}, there is a $J_\varepsilon$-holomorphic sphere $C$ passing through $\iota(0)$ which represents the homology class $[S^2 \!\times\! \pt]$. In particular, there is a holomorphic curve $C' := \iota^{-1}(C) \cap B^{2n+2}_{R-\varepsilon}$ passing through $0 \in \mathbb{C}^{n+1}$ with boundary lying outside $B^{2n+2}_{R-\varepsilon}$. By the monotonicity formula \cite[Theorem~3.2.4]{Lafontaine1994},
$$\pi(R-\varepsilon)^2 \leq \textnormal{area}(C') = \int_{C'} \omega_{std} \leq \int_{C} \omega = \langle \omega , [S^2 \!\times\! \pt] \rangle.$$
As $\varepsilon$ is arbitrary, we are done.
\end{proof}

\begin{rmk}\label{GW_rmk}
If $(X,\omega)$ is semipositive, then \cite[Theorem~5.2.3]{Sikorav94} and the dimension counts in the proof of \cite[Proposition~2.2]{RuanTian1995} show that the evaluation maps~\ref{evaluation_map} with $g=0$ and generic $J$ uniformly tamed by $\omega$ define Borel-Moore pseudocycles. Deforming the almost complex structure via a generic path of uniformly compatible almost complex structures yields a Borel-Moore pseudocycle bordism between the relevant evalution maps. By \cite{CattalaniZinger}, the theory of Borel-Moore pseudocycles up to bordism is naturally equivalent to Borel-Moore homology. Therefore, Gromov-Witten invariants can be defined as Borel-Moore homology classes. By Lemma~\ref{contractibility_lmm}, they do not depend on the choice of uniformly compatible almost complex structure. A strong deformation $(\omega_t)_{t \in [0,1]}$ of symplectic forms lifts to a path of uniformly compatible almost complex structures. If $\omega_t$ is semipositive for all $t \in [0,1]$, then a generic perturbation of this path yields a Borel-Moore bordism between an almost complex structure uniformly compatible with $\omega_0$ and one uniformly compatible with~$\omega_1$. This shows that Gromov-Witten invariants do not change under strong deformations through semipositive symplectic forms. It should be possible to obtain the same conclusions more generally by adapting a standard virtual fundamental class construction to the Borel-Moore homology.
\end{rmk}

\section{Convexity}\label{conv_sec}
In this section, we prove Theorem~\ref{convex_thm}. Our main tool is the use of \textit{convex} almost complex manifolds. This notion is significantly more fragile than that of a uniformly tamed almost complex manifold, although neither subsumes the other. Our aim in this section is to present a simple definition of convexity and its basic functorial properties. In this framework, Theorem~\ref{convex_thm} follows essentially from Proposition~\ref{proper_prop}, Lemma~\ref{conv_fibr_lmm} and \ref{lifting_J_lmm}, and Proposition~\ref{convex_filling_prp}. Of these, only the last two are restricted to dimension $4$. We believe this proof is conceptually simpler and easier to generalize than that in~\cite{Savelyev2025}.\\

We recall that a \textit{compact exhaustion} $(K_i)_{i \in \Z^+}$ of a topological space $X$ is a sequence of compact subsets of $X$ such that $K_i$ is contained in the interior of $K_{i+1}$ and $X = \bigcup K_i$. This implies that every compact subset of $X$ is contained in the interior of some $K_i$. The definition below is a slight modification of those in \cite[Section 0.4]{Gromov1985} and \cite[Section~4.1]{Sikorav94}.

\begin{dfn}\label{convex_dfn}
An almost complex manifold $(X,J)$ is \textit{convex} if there is a compact exhaustion $(K_i)_{i \in \Z^+}$ of~$X$ such that for every $i \in \Z^+$ and every compact connected $J$-holomorphic curve $C$ in~$X$ with $C \cap K_i \neq \emptyset$ and $\partial C \subset K_i$, $C\subset K_i$.
\end{dfn}

In particular, every compact almost complex manifold is convex. If two almost complex structures on $X$ agree outside of a compact subset, then one is convex if and only if the other one is. This is a special case of the following key functoriality property fo convexity, the proof of which is trivial.

\begin{lmm}\label{conv_fibr_lmm}
Let $(X,J_X)$ and $(Y,J_Y)$ be almost complex manifolds, with $(Y,J_Y)$ convex. If there exist a precompact (open) $U \subset X$ and a proper $(J_X,J_Y)$-holomorphic map $f\!: X-U \lra Y$, then $(X,J_X)$ is convex.
\end{lmm}

\begin{eg}\label{Cn_is_convex_eg}
The space $\mathbb{C}^n$ is convex by the Maximum Principle~\cite[Theorem~1.5.1]{Varolin2011}. For example, we could take $K_i = (B_i^2(0))^n$ for the purpose of Definition~\ref{convex_dfn}.
\end{eg}

\begin{eg}\label{conv_eg}
By Lemma~\ref{conv_fibr_lmm} and Example~\ref{Cn_is_convex_eg}, every Stein manifold (a proper holomorphic submanifold of~$\mathbb{C}^n$) is convex.
\end{eg}

\begin{eg}\label{conv_surface_eg}
By \cite[Theorem~12.5.1]{Varolin2011}, each connected component of a Riemann surface (without boundary) is either compact or Stein. Along with Example~\ref{conv_eg}, this implies that every Riemann surface is convex.
\end{eg}

In the following, we let $i$ denote the integrable complex structure on $R^2 = \mathbb{C}$.

\begin{lmm}\label{lifting_J_lmm}
Let $B \subset \R^2$ be a disk. Suppose $\omega$ is a symplectic form on $S^2 \!\times\! (\R^2 - B)$ such that the fibers of the projection map
$$\pi\!: S^2 \!\times\! (\R^2 - B) \lra \R^2 - B$$
are symplectic submanifolds. Then, there is an almost complex structure $J$ on $S^2 \!\times\! (\R^2 - B)$ tamed by $\omega$ such that the map $\pi$ is either $(J,i)$-holomorphic or $(J,-i)$-holomorphic.
\end{lmm}

\begin{proof}
Let $\ker\d\pi$ be the vertical distribution and
$$(\ker\d\pi)^\omega := \bigl\{ v \in T\big(S^2 \!\times\! (\R^2 - B)\big) \,|\, \omega(v,w) = 0 \quad \forall w \in \ker\d\pi \bigr\}$$
be its symplectic complement. The symplectic vector bundles $\ker\d\pi$ and $(\ker\d\pi)^\omega$ over $S^2 \!\times\! (\R^2-B)$ form a direct sum decompositionof $T(S^2 \!\times\! (\R^2-B))$. As a symplectic vector bundle, $(\ker\d\pi)^\omega$ is naturally oriented.
The map $\d\pi$ takes $(\ker\d\pi)^\omega$ onto $T(\R^2 - B)$ and is orientation-preserving with respect to either the orientation given by $i$ or that given by $-i$ on $\R^2 - B$. Without loss of generality, suppose that it is with respect to the orientation given by $i$. Then, define the almost complex structure $J$ on $S^2 \!\times\! (\R^2 - B)$ so it preserves the vertical bundle (with orientation) and is the lift of $i$ on $(\ker\d\pi)^\omega$.
\end{proof}

\begin{prp}\label{convex_filling_prp}
Let $J$ be an almost complex structure on $S^2 \!\times\! \R^2$ which is tamed by a symplectic form. Suppose there is a compact subset $K \subset S^2 \!\times\! \R^2$ such that the projection
$$\pi\!: (S^2 \!\times\! \R^2) - K \lra \R^2$$
is $(J,i)$-holomorphic. Then, for any point $x \in S^2 \!\times\! \R^2$, there exists a $J$-holomorphic sphere representing $[S^2 \!\times\! \pt]$ which passes through $x$.
\end{prp}

\begin{proof}
The class $[S^2 \!\times\! \pt]$ has zero self-intersection. If $x \in \R^2 - \pi(K)$, the fiber $\pi^{-1}(x)$ is an embedded $J$-holomorphic sphere representing $[S^2 \!\times\! \pt]$. By \cite[Theorem~1]{HoferLizanSikorav1997}, the moduli space $\mathfrak{M}_{0,1}([S^2 \!\times\! \pt]; J)$ is a smooth manifold and the evaluation map
$$\ev\!: \mathfrak{M}_{0,1}([S^2 \!\times\! \pt]; J) \lra S^2 \!\times\! \R^2$$
is open. By Lemma~\ref{conv_fibr_lmm} and Example~\ref{Cn_is_convex_eg}, $J$ is convex. By Proposition~\ref{proper_prop}, the evaluation map is proper. Since $S^2 \!\times\! \R^2$ is a connected manifold, it follows that the evaluation map surjects.
\end{proof}

\begin{proof}[{{Proof of Theorem~\ref{convex_thm}}}]
Let $\iota\!: (B^4_R,\omega_{std}) \lra (S^2 \!\times\! \R^{2},\omega)$
be a symplectic embedding. By Lemma~\ref{lifting_J_lmm}, there is an almost complex structure $J$ on \hbox{$S^2 \!\times\! \R^2$} which is tamed by $\omega$ such that the projection map $\pi$ is either $(J,i)$-holomorphic or $(J,-i)$-holomorphic outside of a compact subset. Without loss of generality, we may take $\pi$ to be $(J,i)$-holomorphic.
For $\varepsilon > 0$, let $J_\varepsilon$ be a tamed almost complex structure on $S^2 \!\times\! \R^{2}$ which is equal to the pushforward of the standard integrable complex structure on $\iota(B^4_{R-\varepsilon})$ and to~$J$ outside of a compact set. By Proposition~\ref{convex_filling_prp}, there is a $J_\varepsilon$-holomorphic sphere $C$ passing through $\iota(0)$ which represents the homology class $[S^2 \!\times\! \pt]$. In particular, there is a holomorphic curve $C' := \iota^{-1}(C) \cap B^4_{R-\varepsilon}$ passing through $0 \in \mathbb{C}^{2}$ with boundary lying outside $B^4_{R-\varepsilon}$. By the monotonicity formula \cite[Theorem~3.2.4]{Lafontaine1994},
$$\pi(R-\varepsilon)^2 \leq \textnormal{area}(C') = \int_{C'} \omega_{std} \leq \int_{C} \omega = \langle \omega , [S^2 \!\times\! \pt] \rangle,$$
As $\varepsilon$ is arbitrary, we are done.
\end{proof}

\begin{rmk}
For applications to symplectic geometry, convex almost complex structures have the benefit that taming, as opposed to uniform taming, suffices for properness of families of pseudoholomorphic curves. However, they also have the drawback that it is harder to achieve Fredholm regularity for the Cauchy-Riemann equation. This issue is obviated here and in \cite{Savelyev2025} by the automatic regularity of \cite[Theorem~1]{HoferLizanSikorav1997} in dimension $4$. This appears to present a serious difficulty in generalizing Theorem~\ref{convex_thm} to higher dimensions.
\end{rmk}

\vspace{.3in}

{\it Department of Mathematics, Stony Brook University, Stony Brook, NY 11794\\
spencer.cattalani@stonybrook.edu}

\bibliography{refs}

\end{document}